\documentclass[a4paper, 12pt]{amsart}
\usepackage{amsthm}
\usepackage{xcolor}
\usepackage{enumitem}
\theoremstyle{plain}
\newtheorem{theorem}{Theorem}[section]

\newtheorem{lemma}[theorem]{Lemma}
\newtheorem{corollary}[theorem]{Corollary}
\theoremstyle{definition}
\newtheorem{notation}[theorem]{Notation}
\theoremstyle{remark}
\newtheorem{remark}[theorem]{Remark}

\newcommand{\Iso}{\operatorname{Iso}}
\newcommand{\supp}{\operatorname{supp}}
\newcommand{\id}{\operatorname{id}}

\newcommand{\Aa}{\mathcal{A}}
\newcommand{\Bb}{\mathcal{B}}
\newcommand{\Ii}{\mathcal{I}}

\newcommand{\Iicl}{\overline{\mathcal{I}}}

\numberwithin{equation}{section}

\author[T.M.~Carlsen]{Toke Meier Carlsen}
\address[T.M.~Carlsen]{K\o{}ge, Denmark}
\email{toke.carlsen@gmail.com}

\author{Anna Duwenig}
\address{Department of Mathematics, KU Leuven, Leuven, Belgium}
\email{anna.duwenig@kuleuven.be}

\author[E. Ruiz]{Efren Ruiz}
\address[E. Ruiz]{Department of Mathematics\\University of Hawaii,
Hilo\\200 W. Kawili St.\\
Hilo, Hawaii\\
96720-4091 USA}
\email{ruize@hawaii.edu}

\author[A. Sims]{Aidan Sims}
\address[A. Sims]{School of Mathematics and Applied Statistics\\
University of Wollongong\\
NSW 2522\\
Australia}
\email{asims@uow.edu.au}

\thanks{Anna Duwenig was supported by an FWO Senior Postdoctoral Fellowship (1206124N). This research was supported by Australian Research Council grant DP220101631.}

\title[Norm upper-semicontinuity]{Norm upper-semicontinuity of functions supported on open abelian isotropy in \'etale groupoids
(a corrigendum to \emph{Reconstruction of groupoids and C*-rigidity of dynamical systems}, Adv. Math 390 (2021), 107923)}
\newcommand{\W}{Z} %Anna introduced this stupid macro because Aidan had called a set W while there already was another set by that letter.

\begin{document}

\begin{abstract}
We consider \'etale Hausdorff groupoids in which the interior of the isotropy is abelian. We prove that the norms of the images under regular representations, of elements of the reduced groupoid $C^*$-algebra whose supports are contained in the interior of the isotropy vary upper semicontinuously. This corrects an error in \cite{CRST}.
\end{abstract}

\maketitle

\section*{Introduction}
The purpose of this note is to correct an error in the proof of \cite[Lemma~5.2]{CRST}.

\subsection*{The error}
The error occurs in the final paragraph of the proof of \cite[Lemma~5.2]{CRST}. The fourth line in
that paragraph deals with an element $a$ of the $C^*$-subalgebra
\begin{equation}\label{eq:def A}
A := \{a \in C^*_r(G) : f_a|_{G \setminus \Ii} = 0\} \subseteq C^*_r(G),
\end{equation}
and asserts that
\begin{equation}\label{eq:just plain wrong}
\text{``since $f_a|_{\Iso(G)^\circ_u} = 0$, we have $\|\pi_u(a)\|= 0$.''}
\end{equation}
However, $\pi_u$ has domain $C^*(\Iso(G)^\circ)$, not $A$.
Indeed, \cite[Lemma~5.2]{CRST} is later applied to prove \cite[Corollary~5.3]{CRST}, which
asserts that $A = C^*(\Iso(G)^\circ)$. So the application of $\pi_u$ to an element of $A$
is circular reasoning.

\subsection*{Consequences of the error}
The error invalidates the proof of
\cite[Corollary~5.3]{CRST}, which underpins the main results of \cite{CRST}: it
is used in the paragraph immediately subsequent to \cite[Proposition~6.5]{CRST} to identify
$\big(C_0(G^{(0)})'_{C^*_{c^{-1}(\id_\Gamma)}}\big)_u$ with
$\big(C^*_r(\Iso(c^{-1}(\id_\Gamma))^\circ)\big)_u$. This in turn is used in
the proof of \cite[Proposition~6.5]{CRST} to see that the group
\[
\Iso(c^{-1}(\id_\Gamma))^\circ_u
    \cong \mathcal{U}\big(\big(C^*_r(\Iso(c^{-1}(\id_\Gamma))^\circ)\big)_u\big)
    / \mathcal{U}_0\big(\big(C^*_r(\Iso(c^{-1}(\id_\Gamma))^\circ)\big)_u\big),
\]
can be recovered from $\big(C_0(G^{(0)})'_{C^*_{c^{-1}(\id_\Gamma)}}\big)_u$,
and hence from the triple $(C^*_r(G), C_0(G^{(0)}), \delta_c)$.

\subsection*{Related results}
If $G$ is amenable, then \cite[Theorem~4.2]{BEFPR} implies
\cite[Corollary~5.3]{CRST}. If $G$ contains
a clopen amenable subgroupoid that
contains $\overline{\Iso(G)^\circ}$, we can use the argument of
\cite[Theorem~4.2]{BEFPR} in the same way. If $\Iso(G)^\circ$ is closed, then
the argument of \cite[Proposition~3.12]{DWZ} shows that $A =
C^*(\Iso(G)^\circ)$. However, imposing any of these hypotheses would
substantially weaken the main results of \cite{CRST}.

\subsection*{The correction}
The first three paragraphs of the proof of \cite[Lemma~5.2]{CRST} are correct. In the following paragraph,
the proof that $J_u \subseteq \{a \in A : f_a|_{\operatorname{Iso}(G)^\circ_u}\}$ is correct. The
following four sentences ``For the reverse containment...proves the claim" are invalid as they depend upon
the erroneous assertion~\eqref{eq:just plain wrong} mentioned above. The remainder of the proof is correct
\emph{provided} that it is indeed true that $J_u \supseteq \{a \in A : f_a|_{\operatorname{Iso}(G)^\circ_u}\}$.
Hence Theorem~\ref{thm:phew!} below corrects the proof of \cite[Lemma~5.2]{CRST}, and then \cite[Corollary~5.3]{CRST}
and the subsequent results that depend upon it follow as argued there.

\section{The correction for the proof of \cite[Lemma~5.2]{CRST}}

\begin{notation}
Throughout this note, $G$ is an \'etale groupoid.
For $x\in G^{(0)}$, we write $Gx:=s^{-1}(x)$ and $xG:=r^{-1}(x)$ (note that this differs from
Renault's notation $G_x = s^{-1}(x)$ and $G^x = r^{-1}(x)$),
and we denote by $\Iso(G)$ the isotropy subgroupoid of $G$; that is, $\Iso(G):=
\bigcup_{x \in G^{(0)}} (xG \cap Gx)$. We will assume throughout that the interior
$\Iso(G)^\circ$ of the isotropy, is abelian. We make frequent use
of the set $\Iso(G)^\circ$, its closure $\overline{\Iso(G)^\circ}$, and their
fibres $\Iso(G)^\circ \cap Gx$ and $\overline{\Iso(G)^\circ} \cap Gx$.
Because of the utter failure of commutativity amongst the operations of
adding an overline, a superscripted $\circ$ or a subscripted $x$ to the
expression $\Iso(G)$, we use the following non-standard notation:
\begin{align*}
\Ii &:= \Iso(G)^\circ, & \Iicl &:= \overline{\Iso(G)^\circ},\\
\Ii x &:= \Iso(G)^\circ \cap Gx,\text{ and} & \Iicl x &:= \overline{\Iso(G)^\circ} \cap Gx.
\end{align*}
\end{notation}

We recall that $j \colon C^*_r(G) \to C_0(G)$ is Renault's map extending the
identity map on $C_c(G)$ \cite[Proposition~II.4.2]{Renault} (in \cite{CRST}
this map is denoted $a \mapsto f_a$); if $\lambda_x$ denotes the regular
representation of $C^*_r(G)$ on $\ell^2(Gx)$ for $x \in G^{(0)}$, then $j$ is
defined by $j(a)(\gamma) = \big(\lambda_{s(\gamma)}(a) e_{s(\gamma)} \mid
e_\gamma)$, and satisfies $(\lambda_x(a) e_\gamma \mid e_\delta) =
j(a)(\gamma\delta^{-1})$ for all $x \in G^{(0)}$ and $\gamma,\delta \in Gx$.

For $g \in C_0(G)$, we write $\supp^\circ(g) = \{\gamma \in G : g(\gamma) \not=
0\}$ for the open support of $g$ and $\supp(g) = \overline{\supp^\circ(g)}$ for
its support in the usual sense.

\begin{remark}\label{rmk:products}
We frequently write expressions of the form $\prod_{i \in S} \beta_i$ where
the $\beta_i$ are elements of $G$ and $S$ is an un-ordered indexing set. We
do this only when the $\beta_i$ are pairwise commuting isotropy elements, so
that the product does not depend on the order in which $S$ is listed.
Similarly, we frequently write $\prod_{i \in S} (C_i \cap \Ii)$, where the
$C_i$ are open bisections in $G$, exploiting that $\Ii$ is abelian so that
the product is well defined. If the terms in a product do not necessarily
commute, we write it as $\prod^N_{i=1} F_i$, which always means $F_1 F_2
\cdots F_N$.
\end{remark}

We take the convention that if $\alpha \in G$, then $\alpha^0 =
s(\alpha)$ and if $C \subseteq G$ is an open bisection, then $C^0 = s(C)$.

As detailed above, to correct the proof of \cite[Lemma~5.2]{CRST}, it suffices to establish the
following assertion.

\begin{theorem}\label{thm:phew!}
%Let $G$ be an \'etale groupoid, and suppose that $\Ii$ is abelian.
Let $A = \{a \in C^*_r(G) : j(a)|_{G \setminus \Ii} = 0\} \subseteq C^*_r(G)$. Then for $x \in G^{(0)}$,
\[
\{a \in A : j(a)|_{\Ii x} = 0\} \subseteq J_x := \overline{\{ad : a \in A, d \in
C_0(G^{(0)})\text{ and }d(x) = 0\}}.
\]
\end{theorem}

For the following lemma observe that since $r, s \colon G \to G^{(0)}$ are
continuous, the isotropy $\Iso(G)$ is closed, and in particular each $\Iicl x
\subseteq \Iso(G)$.

\begin{lemma}\label{lem:subgroup}
%Let $G$ be an \'etale groupoid, and suppose that $\Ii$ is abelian.
Suppose
that $\mathcal{C} = \{C_i : i \in S\}$ is a finite family of open bisections
and that $x \in \bigcap_{i \in S} s(C_i)$. For each $i$, let $\gamma_i$
denote the unique element of $C_i x$, and suppose that $\gamma_i \in \Iso(G)$
for all $i \in S$ and that the $\gamma_i$ pairwise commute. Let $H_S$ be the
subgroup of $\Iso(G)x$ generated by $\{\gamma_i : i \in S\}$. Then there is a
neighbourhood $U_S$ of $x$ such that
\begin{enumerate}[label=\textup{(\arabic*)}]
\item\label{it:lem subgroup:integers} for any integers $(m_i)_{i \in S}$ and any $j \in S$ such that
    $\prod_{i \in S} \gamma_i^{m_i} = \gamma_j$, we have $\prod_{i \in S}
    (C_i \cap \Ii)^{m_i} U_S \subseteq C_j$; and
\item\label{it:lem subgroup:hom} for each $y \in U_S$ such that $C_i \cap \Ii y \not= \emptyset$ for
    all $i \in S$, there is a homomorphism $q_y \colon H_S \to \Ii y$ such that
    $q_y(\gamma_i) \in C_i \cap \Ii y$ for all $i \in S$.
\end{enumerate}
\end{lemma}
\begin{proof}
Since $S$ is finite and the $\gamma_i$ commute, $H_S$ is a finitely
generated abelian group. So by the fundamental theorem of finitely generated
abelian groups, there are integers $f, t \ge 0$, elements $\eta_1, \dots,
\eta_{t+f}$ of $H_S$, and strictly positive integers $O_1, \dots, O_t$ such
that $(a_1, \dots, a_{t+f}) \mapsto \prod^{t+f}_{k=1} \eta_k^{a_k}$ is an
isomorphism $\big(\bigoplus_{k \le t} \mathbb{Z}/O_k\mathbb{Z}\big) \oplus
\mathbb{Z}^f \to H_S$. So $H_S$ is generated by $\{\eta_1, \dots,
\eta_{t+f}\}$, and if $\{\zeta_1, \dots, \zeta_{t+f}\}$ are elements of a group
$H'$ that pairwise commute and satisfy $\zeta_k^{O_k} = e_{H'}$ for all $1 \le
k \le t$, then $\eta_k \mapsto \zeta_k$ extends to a homomorphism $H_S \to H'$.

For every $k \le t+f$, since $\eta_k \in H_S$, there is a function $m(k, \cdot) \colon S \to \mathbb{Z}$ such that
\[
    \eta_k = \prod_{i \in S} \gamma_i^{m(k,i)},
\]
using Remark~\ref{rmk:products}. Let
\begin{equation}\label{eq:all intersections}
V_k := \bigcap_{S = \{i_1, \dots, i_{|S|}\}} \prod_{j = 1}^{|S|} C_{i_j}^{m(k,i_j)},
\end{equation}
the intersection over all enumerations of $S$. Since it is an intersection of products of open bisections, $V_k$ is itself an open bisection
neighbourhood of $\eta_k$.

For each $k \le t$, since the unit space of $G$ is open and since
multiplication is continuous, we can shrink $V_k$ to a smaller neighbourhood
of $\eta_k$ such that $V_k^{O_k} \subseteq G^{(0)}$. Shrinking $V_k$ further
by replacing it with the open subset $V_k \cap s^{-1}(V_k^{O_k})$, we can
assume that $V_k^{O_k} = s(V_k)$. Hence $V_k V^{O_k - 1} = s(V_k)$. Since the
open bisections of $G$ form an inverse semigroup $S(G)$
\cite[Proposition~2.2.4]{Paterson}, uniqueness of inverses in inverse
semigroups \cite[page~21]{Paterson} implies that
$V_k^{O_k - 1}$ is the inverse $V_k^* =
\{\alpha^{-1} : \alpha \in V_k\}$
of $V_k$ in $S(G)$.

We have chosen $V_k \owns \eta_k$ for $k \le t+f$ such that for each $k$
and each possible enumeration $S = \{i_1, \dots, i_{|S|}\}$, we have $V_k
\subseteq \prod_{j = 1}^{|S|} C_{i_j}^{m(k,i_j)}$, and such that, in
addition,
\begin{equation}\label{eq:r,s on torsion}
r(V_k) = V_k V_k^* = V_k V_k^{O_k - 1} = V_k^{O_k} = s(V_k)
\quad\text{for }k \le t.
\end{equation}

For each $i \in S$, since $\gamma_i \in H_S$, there is a function $b(i, \cdot)
: \{1, \dots, t+f\} \to \mathbb{Z}$ such that \begin{align}\label{eq:gamma_i in terms of eta_k}
\prod^{t+f}_{k=1}
\eta_k^{b(i,k)}
=
\gamma_i \quad\text{ for each $i \in S$;}
\end{align}
if $i \in S$ satisfies $\gamma_i = x$, then we take $b(i,k) = 0$ for all $k$. For any $i \in S$,
both $C_i$ and $\prod^{t+f}_{k=1} V_k^{b(i,k)}$ are open
bisection neighbourhoods of $\gamma_i$; so we can further shrink each $V_k$
to a smaller neighbourhood of $\eta_k$ to ensure that
\begin{equation}\label{eq:prod of V_k in C_i}
    \prod^{t+f}_{k=1} V_k^{b(i,k)} \subseteq C_i\quad\text{ for each $i \in S$.}
\end{equation}
Now fix $y\in \bigcap_{i\in S}s(C_i)$, so there is a unique element in $C_i y$ for each $i\in S$, and fix $k\leq t+f$. If $y$ satisfies $C_i y \subseteq \Ii$ for each $i \in S$ satisfying $m(k,i) \not= 0$, then
the unique elements of
$(C_i y)^{m(k,i)}$ for $i\in S$ pairwise commute because $\Ii$ is abelian; so $\prod_{i\in S}(C_i y)^{m(k,i)}$ is well-defined, and is a singleton. For any enumeration $S=\{i_{1},\ldots,i_{|S|}\}$, the set $\prod_{j=1}^{|S|} C_{i_j}^{m(k,i_j)}$ contains $\prod_{i\in S}(C_i y)^{m(k,i)}$. Since this set is a bisection, we deduce that
\begin{equation}\label{eq:BiI is V}
    \prod_{i\in S}(C_i y)^{m(k,i)}
    =
    \Bigl(\bigcap_{S=\{i_1,\ldots,i_{|S|}\}}\prod_{j=1,\ldots,|S|} C_{i_j}^{m(k,i_j)}\Bigr)y
    \overset{\eqref{eq:all intersections}}{=} V_{k}y.
\end{equation}

Since each $\eta_k \in V_k \cap H_S$, we have $x \in \bigcap^{t+f}_{k=1}
s(V_k)$, so~\eqref{eq:prod of V_k in C_i} implies that there is a neighbourhood $U_S \subseteq \big(\bigcap^{t+f}_{k=1}
s(V_k)\big) \cap \big(\bigcap_{i \in S} s(C_i)\big)$ of $x$ such that
\begin{equation}\label{eq:Vk small enough}
\Big(\prod^{t+f}_{k=1} V_k^{b(i,k)}\Big) U_S = C_i U_S\quad\text{ for each $i \in S$.}
\end{equation}
For each $j \in S$ satisfying $\gamma_j = x$, we took $b(j,k) = 0$ for all $k$; hence for such $j$, Equation~\eqref{eq:Vk small enough} becomes
$U_S = \Big(\prod^{t+f}_{k=1} s(V_k)\Big) U_S =  C_j U_S$. In particular, $C_j U_S \subseteq G^{(0)}$.

We show that $U_S$ satisfies \ref{it:lem subgroup:integers}~and~\ref{it:lem subgroup:hom}.

For~\ref{it:lem subgroup:integers}, suppose that $\prod_{i \in S} \gamma_i^{m_i} = \gamma_j$. Fix $y \in
U_S$, and suppose that $\prod_{i\in S}(C_i\cap \Ii)^{m_i}y\neq \emptyset$; we must show that this set is contained in $C_{j}$.
As $\prod_{i\in S}(C_i\cap \Ii)^{m_i}y\neq \emptyset$, each $(C_i\cap \Ii)^{m_i}y$ is non-empty. Since
$C_i\cap\Ii\subseteq \Iso(G)$, we have $(C_i\cap \Ii)^{m_i}y = (C_i\cap \Ii y)^{m_i}$, and so for each $i\in S$,
the set $C_i\cap \Ii y$ is non-empty; we let $\beta_i$ denote its unique element so that
\[
\prod_{i\in S}(C_i\cap \Ii)^{m_i}y = \Big\{\prod_{i \in S} \beta_i^{m_i}\Big\}.
\]

Since each $\eta_k$ is an element of the abelian group $H_{S}$, Equation~\eqref{eq:gamma_i in terms of eta_k} implies that
$\gamma_i^{m_i} = \prod^{t+f}_{k=1} \eta_k^{m_i b(i,k)}$ for all $i$. Since, by assumption, $\prod_i \gamma_i^{m_i} = \gamma_j$,
\[
\prod^{t+f}_{k=1} \eta_k^{\sum_{i\in S} m_i b(i,k)}
    = \prod_{i \in S} \Big(\prod^{t+f}_{k=1} \eta_k^{m_i b(i,k)}\Big)
    = \prod_{i \in S} \gamma_i^{m_i}
    = \gamma_j
    = \prod^{t+f}_{k=1} \eta_k^{b(j,k)}.
\]
Since $(a_1, \dots, a_{t+f}) \mapsto \prod^{t+f}_{k=1} \eta_k^{a_k}$ is an
isomorphism $\big(\bigoplus_{k \le t} \mathbb{Z}/O_k\mathbb{Z}\big) \oplus
\mathbb{Z}^f \to H_S$, there are integers $x_k, k \le t$ such that
\begin{align}
b(j,k)&= x_k O_k + \sum_{i \in  S } m_i b(i,k)\qquad\text{for $k \le t$, and}\label{eq:klet case} \\
b(j,k) &= \sum_{i \in  S } m_i b(i,k)\qquad\text{for $k > t$.}\label{eq:k>t case}
\end{align}
We have
\begin{equation}\label{eq:torsion}
V_k^{x_k O_k} = (V_k^{O_k})^{x_k} = s(V_k)^{x_k} = s(V_k) \supseteq U_S\quad\text{ for $k \le t$.}
\end{equation}
Since $C_{i}y\subseteq \Ii$ for all $i\in S$ and since $\Ii$ is abelian, for
any enumeration $S = \{i_1, \dots, i_{|S|}\}$ of $S$ and any integers
$n_{1},\ldots,n_{|S|}$, we have
\[
\Big(\prod^{|S|}_{j=1} C_{i_j}^{n_{i_j}}\Big)y
    = \prod_{i\in S} (C_{i}y)^{n_i}
    = \Big\{ \prod_{i\in S} \beta_{i}^{n_i}\Big\}.
\]
So~\eqref{eq:all intersections} and that $\Ii$ is closed under multiplication
imply that for each $k\le t+f$, defining $\zeta_k := \prod_{i \in S}
\beta_i^{m(k,i)} \in  \Ii y$, we have
\begin{equation}\label{eq:prod beta i}
V_k y = \{\zeta_k\}.
\end{equation}
Since $\Ii$ is abelian, the $\zeta_k$ pairwise
commute. Since $y\in U_S$, for each $i\in S$,
\begin{align}\label{eq:beta_i in terms of zetas}
    \{\beta_{i}\}
    =
    C_{i} y
    \overset{\eqref{eq:Vk small enough}}{=}
    \Big(\prod^{t+f}_{k=1} V_k^{b(i,k)}\Big) y
    =
    \prod^{t+f}_{k=1} (V_k y)^{b(i,k)}
    \overset{\eqref{eq:prod beta i}}{=}
    \Big\{ \prod^{t+f}_{k=1} \zeta_k^{b(i,k)}\Big\}.
\end{align}
Using again that $\Ii$ is abelian, we obtain
\begin{align}\label{eq:prod of betas in terms of zetas}
    \prod_{i \in  S } \beta_i^{m_i}
    =
    \prod_{i \in  S } \Big( \prod^{t+f}_{k=1} \zeta_k^{b(i,k)}\Big)^{m_i}
    =
    \prod^{t+f}_{k=1} \zeta_k^{\sum_{i\in S} m_i b(i,k)}
    % \in
    % \Ii y
    .
\end{align}
Consequently,
\begin{align*}
\prod_{i \in  S } (C_i\cap\Ii)^{m_i}y
    &= \Big\{ \prod^{t+f}_{k=1} \zeta_k^{\sum_{i\in S} m_i b(i,k)} \Big\}\\
    &\subseteq \Big(\prod^{t+f}_{k = 1} V_k^{\sum_{i \in  S } m_i b(i,k)}\Big) U_S
    = \Big(\prod^t_{k=1} V_k^{\sum_{i \in  S } m_i b(i,k)}\Big) \Big(\prod^{t+f}_{k=t+1} V_k^{\sum_{i \in  S } m_i b(i,k)}\Big) U_S.
\end{align*}
By~\eqref{eq:torsion}, for each $k \le t$ we have
\[
V_k^{\sum_{i \in  S } m_i b(i,k)}
    = V_k^{\sum_{i \in  S } m_i b(i,k)} V_k^{x_k O_k}
    = V_k^{x_k O_k + \sum_{i \in  S } m_i b(i,k)}.
\]
Hence
\[
\prod_{i \in S } (C_i\cap\Ii)^{m_i}y
    = \Big(\prod^t_{k=1} V_k^{x_k O_k + \sum_{i \in  S } m_i b(i,k)}\Big) \Big(\prod^{t+f}_{k=t+1} V_k^{\sum_{i \in  S } m_i b(i,k)}\Big) U_S.
\]
Now applying~\eqref{eq:klet case} to the first product and~\eqref{eq:k>t
case} to the second, we obtain
\[
\prod_{i \in S } (C_i\cap\Ii)^{m_i}y
        = \Big(\prod^t_{k=1} V_k^{b(j,k)}\Big) \Big(\prod^{t+f}_{k=t+1} V_k^{b(j,k)}\Big) U_S
        = \prod^{t+f}_{k=1} V_k^{b(j,k)}.
\]
Hence~\eqref{eq:Vk small enough} gives
\[
\prod_{i \in  S } (C_i\cap\Ii)^{m_i}y
    = C_j U_S.
\]
For~\ref{it:lem subgroup:hom}, suppose that $y \in U_S$ and that $C_i \cap
\Ii y \not= \emptyset$ for each $i \in S$, and denote its unique element by
$\beta_i$. We must show that there is a homomorphism $q_y \colon H_S \to \Ii
y$ such that $q_y(\gamma_i) = \beta_i$ for all $i \in S$. The elements
$\{\zeta_k\} = V_k y$ defined at~\eqref{eq:prod beta i} belong to the group
$\Ii y$ and pairwise commute because $\Ii$ is abelian. For $k \le t$, we have
$\zeta_k^{O_k} \in V_k^{O_k}$ by~\eqref{eq:prod beta i}, and $V_k^{O_k}
\subseteq G^{(0)}$ by definition of $O_k$, so $\zeta_k^{O_k} = y$ is the
identity element of $\Ii y$. So the universal property of $H_S$ yields a
homomorphism $q_y \colon H_S \to \Ii y$ such that $q_y(\eta_k) = \zeta_k$ for
all $k \le t + f$. Using~\eqref{eq:gamma_i in terms of eta_k} at the
first equality and~\eqref{eq:beta_i in terms of zetas} at the third, we see
that, for each $i \in S$,
\[
q_y(\gamma_i)
    = \prod^{t+f}_{k=1} q_y(\eta_k)^{b(i,k)}
    = \prod^{t+f}_{k=1} \zeta_k^{b(i,k)}
    = \beta_{i}.\qedhere
\]
\end{proof}

\begin{corollary}\label{cor:all the subgroups}
%Let $G$ be an \'etale groupoid, and suppose that $\Ii$ is abelian.
Suppose
that $\{C_i : i \in I_0\}$ is a finite family of open bisections and that $x
\in \bigcap_{i \in I_0} s(C_i)$. For each $i$, let $\gamma_i$ denote the
unique element of $C_i x$, and suppose that $\gamma_i\in\Iso(G)$ for all $i
\in I_0$. Then there exists a neighbourhood $U_0$ of $x$ such that
\begin{enumerate}[label=\textup{(\roman*)}]
\item\label{it:cor subgps:K_y} for each $y \in U_0$, the set $\{\gamma_i : C_i  \cap \Ii y \not=
    \emptyset\}$ is a commutative subset of $G$ and the subgroup $K_y$ of
    $\Iso(G)x$ that it generates is contained in $\Iicl x$;
\item\label{it:cor subgps:integers} for any $y\in U_0$, any $j \in I_0$,
    any nonemtpy $S\subseteq \{i\in I_0 : C_i \cap \Ii y \not=
    \emptyset\}$, and any integers $(m_i)_{i \in S}$ such that $\prod_{i
    \in S} \gamma_i^{m_i} = \gamma_j$, we have $\prod_{i \in S} (C_i \cap
    \Ii)^{m_i} U_0 \subseteq C_j$; and
\item\label{it:cor subgps:hom} for each $y \in U_0$, there is a
    homomorphism $q_y$ from the subgroup $K_y$ of $\Iso(G) x$ generated by
    $\{\gamma_i : C_i \cap \Ii y \not= \emptyset\}$ to $\Ii y$ such that
    for each $i \in I_0$ satisfying $C_i \cap \Ii y \not= \emptyset$, we
    have $q_y(\gamma_i) \in C_i y$.
\end{enumerate}
\end{corollary}
\begin{proof}
Fix a subset $S$ of $I_0$. If there exists a neighbourhood $W$ of $x$ such
that for every $y \in W$, there exist $i \in S$ such that $C_i \cap \Ii y =
\emptyset$, let $V_S$ be such a neighbourhood; otherwise put $V_S = G^{(0)}$.
Let $V := \bigcap_{S \subseteq I_0} V_S$. Then $V$ is an open neighbourhood
of $x$. For $y\in V$, let
\[
    S_y := \{i \in I_0: C_i \cap \Ii y \not= \emptyset\}.
\]
By definition of the sets $V_S$,
\begin{equation}\label{eq:Toke is clever}
\parbox{0.9\textwidth}{if $y \in V$ and $S\subseteq S_y$, then every neighbourhood
of $x$ contains a point $z$ such that $C_i \cap \Ii z\not= \emptyset$ for
all $i \in S$.}
\end{equation}
This implies that for each $y \in V$ with $S_y\neq\emptyset$, there is a net
$(z_\lambda)_{\lambda \in \Lambda_y}$ converging to $x$ such that $C_i \cap
\Ii z_\lambda \not= \emptyset$ for all $i \in S_y$ and all $\lambda \in
\Lambda_y$; for each $i \in S_y$ and $\lambda \in \Lambda_y$ we define
$\beta_{i,\lambda}$ be the unique element of $C_i \cap \Ii z_\lambda$. Since
each $C_i$ is a bisection and hence each $s|_{C_i}$ is a homeomorphism, for each $i \in S_y$,
\begin{equation}\label{eq:beta_lambda to gamma}
C_i
\cap \Ii z_\lambda \ni \beta_{i,\lambda} \to \gamma_i.
\end{equation}
We claim that $V$ satisfies~\ref{it:cor subgps:K_y}; it then follows that any
open subset $U_0$ of $V$ also satisfies~\ref{it:cor subgps:K_y}. Fix $y \in
V$; we claim that $\{\gamma_i : i\in S_y\}$ is a set of pairwise commuting
elements of $G$. If $S_y$ is empty, the claim is vacuous, so suppose that
$S_y\neq\emptyset$ and fix $i,j \in S_y$. Then for each $\lambda \in
\Lambda_y$ we have $\beta_{i, \lambda}, \beta_{j,\lambda} \in \Ii z_\lambda$,
and in particular $\beta_{i, \lambda} \beta_{j,\lambda} = \beta_{j, \lambda}
\beta_{i,\lambda}$ because $\Ii$ is abelian. So using~\eqref{eq:beta_lambda
to gamma} twice, we see that
\[
\gamma_i \gamma_j
    = \lim_{\lambda \in \Lambda_y} \beta_{i,\lambda}\beta_{j,\lambda}
    = \lim_{\lambda \in \Lambda_y} \beta_{j,\lambda}\beta_{i,\lambda} = \gamma_j\gamma_i.
\]
Hence $\{\gamma_i : i\in S_y \}$ is a set of pairwise commuting elements. For
the remaining assertions in~\ref{it:cor subgps:K_y}, suppose that $\gamma$
belongs to $K_y$. Then $\gamma$ has the form $\prod_{i \in S_y}
\gamma_i^{m_i}$, and so $\lim_\lambda\prod_{i \in S_y}
\beta_{i,\lambda}^{m_i} =\gamma$. Since $\Ii$ is closed under multiplication,
we have $\prod_{i \in S_y} \beta_{i,\lambda}^{m_i} \in \Ii$ for all
$\lambda$, and so $\gamma \in \Iicl x$. Hence $V$ satisfies~\ref{it:cor
subgps:K_y} as claimed.

Let
\[
    \mathcal{S} := \{S \subseteq I_0 : S \not= \emptyset \text{ and there exists }y \in V \text{ such that } S\subseteq S_y \}.
\]
First suppose that $\mathcal{S} = \emptyset$. Then $S_y=\emptyset$ for all
$y\in V$. Hence Condition~\ref{it:cor subgps:integers} is vacuously true for
any $U_0\subseteq V$. Thus Condition~\ref{it:cor subgps:hom} reduces to
asserting the existence of the trivial homomorphism from $\{x\}$ to $\Ii y$,
and so we are done. Now suppose that $\mathcal{S}$ is nonempty. For each $S
\in \mathcal{S}$, since $V$ satisfies Condition~\ref{it:cor subgps:K_y}, we
can apply Lemma~\ref{lem:subgroup} to the bisections $\{C_i V : i \in S\}$ to
obtain a neighbourhood $U_S \subseteq V$ of $x$ satisfying
Conditions~\ref{it:lem subgroup:integers} and~\ref{it:lem subgroup:hom} of
that result. Let $U_0 := \bigcap_{S \in \mathcal{S}} U_S$, which is an open
neighbourhood of $x$.

Since $U_0$ is an open subset of $V$, we saw above that it satisfies
Condition~\ref{it:cor subgps:K_y}. To see that it satisfies
Condition~\ref{it:cor subgps:integers}, fix a point $y\in U_0$, an element $j
\in I_0$, and a nonempty set $S\subseteq \{i\in I_0 : C_i \cap \Ii y \not=
\emptyset\}$, and suppose that $(m_i)_{i \in S}$ are integers satisfying
$\prod_{i \in S} \gamma_i^{m_i} = \gamma_j$. We must show that $\prod_{i \in
S} (C_i \cap \Ii)^{m_i} U_0 \subseteq C_j$.

Let $(z_\lambda)_{\lambda \in \Lambda_y}$ and $\beta_{i,\lambda}\in C_i \cap
\Ii z_\lambda$ for $i\in S_y$ be as above. By~\eqref{eq:beta_lambda to
gamma}, we have $\lim_\lambda \prod_{i \in S} \beta_{i,\lambda}^{m_i} =
\prod_{i \in S} \gamma_i^{m_i} = \gamma_j$. Since $C_j$ is a neighbourhood of
$\gamma_j$, we have $\prod_{i \in S} \beta_{i,\lambda}^{m_i} \in C_j \cap \Ii
z_\lambda$ for large $\lambda$. In particular, for large $\lambda$, the point
$z_{\lambda} \in V$ satisfies $C_i \cap \Ii z_{\lambda} \not= \emptyset$ not
only for $i\in S$ but also for $i=j$. Hence $S \cup \{j\} \in \mathcal{S}$.
Since $U_0 \subseteq U_{S \cup \{j\}}$, Condition~\ref{it:lem
subgroup:integers} of Lemma~\ref{lem:subgroup} for $S \cup \{j\}$ implies
that $\prod_{i \in S} (C_i \cap \Ii)^{m_i} U_0 \subseteq C_j$. Hence $U_0$
satisfies Condition~\ref{it:cor subgps:integers}.

Finally, to see that $U_0$ satisfies Condition~\ref{it:cor subgps:hom}, fix
$y \in U_0 \subseteq U_{S_{y}}$. Condition~\ref{it:lem subgroup:hom} of
Lemma~\ref{lem:subgroup} for $S_{y}$ gives a homomorphism from $H_{S_{y}} =
K_y \le \Iicl x$ to $\Ii y$ with the desired property.
\end{proof}

We will need a simple technical result regarding expressions for elements of
$C_c(G)$ as sums of functions supported on bisections.

\begin{lemma}\label{lem:can rewrite g}
Let $G$ be an \'etale groupoid (for this lemma, $\Ii$ need not be abelian) and fix $g \in C_c(G)$ and $x \in G^{(0)}$.
There is a finite set $\Bb$ of open bisections of $G$ and for each $B \in
\Bb$ a function $g_B \in C_c(B) \subseteq C_c(G)$ such that $g = \sum_{B \in
\Bb} g_B$ and such that if $B, B' \in \Bb$ satisfy $B \cap B' \cap
Gx
\not=
\emptyset$, then $B = B'$.
\end{lemma}
\begin{proof}
The standard partition-of-unity argument shows that there is a finite set
$\Aa$ of open bisections and for each $A \in \Aa$ a function $f_A \in C_c(A)$
such that $g = \sum_{A \in \Aa} f_A$.

For each of the finitely many $\gamma \in \big(\bigcup \Aa\big) x$, let
$\Aa_\gamma := \{A \in \Aa : \gamma \in A\}$. Then
%for $\gamma \in \big(\bigcup \Aa\big) x$,
the set
$B(\gamma) := \bigcap \Aa_\gamma$ is an open bisection containing $\gamma$. Fix
$h_\gamma \in C_c(B(\gamma), [0,1])$ such that $h_\gamma$ is identically~1 on a
neighbourhood of $\gamma$. Using $\cdot$ to denote pointwise multiplication in
$C_c(G)$, define
\[
g_{B(\gamma)} := \sum_{A \in \Aa_\gamma} h_\gamma\cdot f_A \in C_c(G).
\]
For each $A \in \Aa_\gamma$, the set $B(A) :=	A\setminus ( A\cap Gx) =
A\setminus\{\gamma\}$ is an open bisection that does not meet $Gx$; we put
$g_{B(A)} = (1 - h_\gamma)\cdot f_A$. Since $h_\gamma$ is identically~1 on a
neighbourhood of $\gamma$, each $g_{B(A)}$ vanishes on a neighbourhood of
$\gamma$ and so by choice of $f_A$, we have $g_{B(A)} \in C_c(B(A))$. By
construction, $\sum_{A \in \Aa_\gamma} f_A = g_{B(\gamma)} + \sum_{A \in
\Aa_\gamma} g_{B(A)}$. Since each $A \in \Aa$ is a bisection, the
$\Aa_\gamma$ are mutually disjoint, so
\[
\sum_{A \in \Aa, Ax \not= \emptyset} f_A
    = \sum_{\gamma \in (\bigcup \Aa)x}\Big(\sum_{A \in \Aa_\gamma} f_A\Big)
    = \sum_{\gamma \in (\bigcup \Aa)x} \Big(g_{B(\gamma)} + \sum_{A \in \Aa_\gamma} g_{B(A)}\Big).
\]
For $A \in \Aa$ such that $A x = \emptyset$, let $B(A) = A$ and $g_A = f_A$.
Let $\Bb := \{B(\gamma) : \gamma \in \big(\bigcup \Aa\big) x\} \cup \{B(A) :
A \in \Aa\}$. Since $\Aa$ is finite, so is $\Bb$. We have
\[
g = \sum_{A \in \Aa} f_A
    =
    \Big(\sum_{A \in \Aa, Ax = \emptyset} f_A\Big)
    +
    \sum_{\gamma \in (\bigcup \Aa)x} \Big(g_{B(\gamma)} + \sum_{A \in \Aa_\gamma} g_{B(A)}\Big)
    = \sum_{B \in \Bb} g_B.
\]
Suppose that $B, B' \in \Bb$ satisfy $B \cap B' \cap
Gx
 \not= \emptyset$.
Then in particular $B \cap Gx
$ and $B' \cap Gx
$ are nonempty. By
construction, $B(A) \cap Gx
 = \emptyset$ for all $A \in \Aa$, so there exist
$\gamma,\gamma' \in \big(\bigcup \Aa\big) x$ such that $B = B(\gamma)$ and $B'
= B(\gamma')$. Hence
$B \cap Gx = \{\gamma\}$ and $B' \cap Gx =
\{\gamma'\}$, so $B \cap B' \cap G x \not= \emptyset$ forces $\gamma =
\gamma'$  and thus
$B = B(\gamma) = B(\gamma') = B'$.
\end{proof}

\begin{lemma}\label{lem:editing g}
%Let $G$ be an \'etale groupoid, and suppose that $\Ii$ is abelian.
Fix $x \in
G^{(0)}$, $g \in C_c(G)$ and $\varepsilon > 0$. There is a neighbourhood $U$ of
$x$ such that for each $y \in U$ there are an abelian subgroup $K_y \le
\Iso(G)x$ with $K_y \subseteq \Iicl x$ and a homomorphism $q_y \colon K_y \to \Ii y$
such that
\begin{enumerate}[label=\textup{(\alph*)}]
\item\label{it:lem editing g:subset} $\Ii y \cap \supp^\circ (g) \subseteq
    q_y(K_y \cap \supp(g))$;
\item\label{it:lem editing g:injective} $q_y$ is injective on $K_y \cap \supp(g)$; and
\item\label{it:lem editing g:epsilon} $|g(q_y(\gamma)) - g(\gamma)| < \varepsilon$ for all $\gamma \in K_y
    \cap \supp(g)$.
\end{enumerate}
\end{lemma}
\begin{proof}
Since $g \in C_c(G)$, Lemma~\ref{lem:can rewrite g} implies that we can write
$g = \sum^n_{i=1} g_i$ so that for each $i$ there is an open bisection $C_i$
such that $g_i \in C_c(C_i)$, and so that if $C_i \cap C_j \cap Gx \not=
\emptyset$, then $i = j$.

Let
\[
I_0 := \{i \le n : \supp(g_i) \cap \Iicl x \not= \emptyset\}.
\]
Then for $i \in I_0$, since $\supp(g_i) \subseteq C_i$, it follows that $C_i
\cap \Iicl x \not= \emptyset$. Since $C_i$ is a bisection, $C_i x$ is a
singleton; we write $\gamma_i$ for the unique element of $\supp(g_i)x
\subseteq C_i x$. By choice of the $C_i$ and $g_i$ in the preceding
paragraph, the map $i \mapsto \gamma_i$ is an injection from $I_0$ to $\Iicl
x$.

Since $G$ is Hausdorff, the $\gamma_i$ can be separated by mutually disjoint
open neighbourhoods. For each $i$, sets of the form $C_i V$, where $V$ is an
open neighbourhood of $s(\gamma_i) = x$ are a neighbourhood base at
$\gamma_i$. Since $i \mapsto \gamma_i$ is injective, it follows that there is
a neighbourhood $W$ of $x$ such that the bisections $C_i W, i \in I_0$ are
mutually disjoint \label{page:C_i W mutually disjoint} and such that whenever
$i \in I_0$ satisfies $g_i(\gamma_i) \not= 0$ we have $0 \not\in g_i(C_i W)$.
Since $g=\sum_{i=1}^{n}g_i$ and $\{\gamma_i\}=\supp(g_i)x\subseteq C_i$, this
implies in particular that
\begin{equation}\label{eq:gamma_i in supp(g)}
    \gamma_i\in \supp(g) \text{ for all }i\in I_0.
\end{equation}
Corollary~\ref{cor:all the subgroups} applied to these
bisections gives a neighbourhood $U_0 \subseteq W$ of $x$ satisfying
Conditions \ref{it:cor subgps:K_y}, \ref{it:cor subgps:integers}, \ref{it:cor
subgps:hom} of that corollary.

For each $i \in I_0$, since $g_i$ is continuous and since $s \colon C_i \to s(C_i)
\subseteq G^{(0)}$ is a homeomorphism, there is a neighbourhood $U_i$ of $x$
such that $\sup_{\eta \in C_i U_i} |g_i(\eta) - g_i(\gamma_i)| <
\varepsilon$.
Let
\[
    U_{I_0} := \bigcap_{i \in I_0} U_i.
\]
Then $U_{I_0}$ is an open neighbourhood of $x$ satisfying
\begin{equation}\label{eq:norm control}
	\sup_{\eta \in C_i U_{I_0}} |g_i(\eta) - g_i(\gamma_i)| < \varepsilon
	\text{ for every } i\in I_0.
\end{equation}

Fix $i \in \{1, \dots, n\} \setminus I_0$. We claim that there is a
neighbourhood $U_i$ of $x$ such that $\supp^\circ(g_i) U_i \cap \Ii =
\emptyset$. To see this, consider
two cases: either $x \in s(\supp(g_i))$ or not. If $x \in s(\supp(g_i))$, let
$\gamma$ be the unique element of $\supp(g_i) \subseteq C_i$ such that
$s(\gamma) = x$. Since $\supp(g_i) \cap \Iicl x = \emptyset$, we have $\gamma
\not\in \Iicl x$. Hence there is an open set
$\W \subseteq G$ containing $\gamma$ such that
$\W \cap \Iicl = \emptyset$. Let $U_i := s(\W \cap C_i)$. This is an open set
because $s$ is open, and we have $\supp^\circ(g_i) U_i \subseteq C_i U_i = \W
\cap C_i $ because $C_i$ is a bisection. Since $\W \cap \Ii \subseteq
\W \cap \Iicl = \emptyset$, this set $U_i$ does the job. Now suppose that $x
\not\in s(\supp(g_i))$. Since $s|_{C_i}$ is a homeomorphism and $\supp(g_i)
\subseteq C_i$, $s(\supp(g_i))$ is closed. So there is a neighbourhood $U_i$ of
$x$ such that $U_i \cap s(\supp(g_i)) = \emptyset$. So $\supp^\circ(g_i) U_i$
is empty, and in particular does not intersect $\Ii$. This proves the claim.

Let
\[
    U_{I_0^c} := \bigcap_{i \in \{1 ,\dots, n\} \setminus I_0} U_i.
\]
Then $U_{I_0^c}$ is an open neighbourhood of $x$ (we use the convention that
$U_{I_0^c}=G^{(0)}$ if $\{1 ,\dots, n\} \setminus I_0=\emptyset$). So for all
$i \in \{1, \dots, n\} \setminus I_0$, we have $\supp^\circ(g_i) U_{I_0^c}
\cap \Ii = \emptyset$.

Now let
\[
U := U_0 \cap U_{I_0} \cap U_{I_0^c}.
\]
We show that this $U$ has the desired properties. Fix $y \in U$. Let
\[
S:= \{i \in I_0: C_i \cap \Ii y \not= \emptyset\},
\]
and let $K_y$ be the subgroup of $\Iso(G) x$ generated by
$\{\gamma_i : i\in S\}$. Since $U \subseteq U_0$,
Condition~\ref{it:cor subgps:K_y} of Corollary~\ref{cor:all the subgroups}
guarantees that $K_y$ is an abelian subgroup of $\Iso(G) x$ and is contained
in $\Iicl x$, and Condition~\ref{it:cor subgps:hom} of Corollary~\ref{cor:all
the subgroups} gives a homomorphism $q_y \colon K_y \to \Ii y$ such that
$q_y(\gamma_i) \in C_i y$ for all $i\in S$.

For Property~\ref{it:lem editing g:subset}, suppose that $\beta \in \Ii y$
satisfies $g(\beta) \not= 0$. Then $g_i(\beta) \not= 0$ for some $i \in \{1,
\dots, n\}$. As $y\in U\subseteq U_{I_0^c}\subseteq U_i$, this means
$\{\beta\}= \supp^\circ (g_i)U_i\cap \Ii y \subseteq C_i \cap \Ii y$. By our
choice of the sets $U_j$ for $j \in\{1,\ldots,n\}\setminus I_0$, this implies
that $i \in I_0$ and hence $i\in S$. Thus, $\gamma_i$ is one of the
generators of $K_y$. Since $q_y(\gamma_i)$ is the unique element of $C_i y$,
it must coincide with $\beta$, and we deduce with~\eqref{eq:gamma_i in
supp(g)} that $q_y(\gamma_i) \in q_y(K_y \cap \supp(g))$ as required.

For~\ref{it:lem editing g:injective}, note first that this is trivial if
$S=\emptyset$, since then the set $K_y\cap \supp(g)$ is at most a singleton.
We may therefore assume that $S\neq\emptyset$.

We claim that, if $i\in I_0$ satisfies $\gamma_i\in K_y$, then
$q_{y}(\gamma_i)\in C_i y$ (so in particular $i\in S$). To see this, write
$\gamma_i$ as a product of the generators of $K_{y}$;
that is, fix integers $m_l, l\in S$ such that
$\gamma_{i}=\prod_{l\in S}\gamma_l^{m_l}$. By construction of $S$, we may
apply Condition~\ref{it:cor subgps:integers} of Corollary~\ref{cor:all the
subgroups} to it:
 \begin{align*}
         \Ii y\ni q_{y}(\gamma_i)
        &=
        \prod_{l\in S} q_{y}( \gamma_l)^{m_l}
        &&\text{by choice of }(m_l)_l
        \\
        &\in
        \prod_{l \in S} (C_l \cap   \Ii y)^{m_l}
        &&\text{by Corollary~\ref{cor:all the subgroups}\ref{it:cor subgps:hom} and choice of $S$}
        \\
        &\subseteq
        \prod_{l \in S} (C_l \cap \Ii)^{m_l} U_0
        &&\text{as }y\in U \subseteq U_0
        \\
        &\subseteq C_i
        &&\text{by Corollary~\ref{cor:all the subgroups}\ref{it:cor subgps:integers} and choice of $(m_{l})_l$}.
\end{align*}
This establishes the claim.

Now, suppose that $\gamma, \gamma'$ are distinct elements of $K_y \cap
\supp(g)$. Since $K_y \subseteq \Iicl x$ and $\supp(g)\subseteq
\bigcup_{i\in\{1,\ldots,n\}} C_i$, it follows from the opening paragraph of
the proof that $\gamma = \gamma_i$ and $\gamma' = \gamma_j$ for distinct
$i, j \in I_0$. By the above argument, we have $q_{y}(\gamma)\in C_i y$ and
$q_{y}(\gamma')\in C_j y$. Since we chose the neighborhood $W$ of $x$ so that
the bisections $C_i W$, $i \in I_0,$ are mutually disjoint (page
\pageref{page:C_i W mutually disjoint}) and since $y\in U\subset U_0\subseteq
W$, we have $C_i y \cap C_j y = \emptyset$. In other words,
$q_{y}(\gamma)\neq q_{y}(\gamma')$, as claimed.

It remains to establish~\ref{it:lem editing g:epsilon}. Suppose that $\gamma
\in K_y \cap \supp(g)$. As in the argument for~\ref{it:lem editing
g:injective} above, $\gamma = \gamma_j \in C_j$ for some $j \in I_0$, and the
claim above shows that $q_y(\gamma) \in C_j y$. By our choice of the sets
$U_i$ for $i \not\in I_0$ and since $y\in U\subseteq U_i$, we have
$g(q_y(\gamma)) = \sum^n_{i=1} g_i(q_y(\gamma)) = \sum_{i \in I_0}
g_i(q_y(\gamma))$. As in the preceding paragraph, $C_i U_0 \cap C_j U_0 =
\emptyset$ for $i \in I_0 \setminus \{j\}$, so $q_y(\gamma) \not\in C_i$ for
$i \in I_0 \setminus \{j\}$, and hence $g_i(q_y(\gamma)) = 0$ for $i \in I_0
\setminus \{j\}$. Hence $g(q_y(\gamma)) = g_j(q_y(\gamma))$. In the case
$y=x$, so that $q_y=\mathrm{id}$, this proves $g(\gamma) = g_j(\gamma)$.
Since $s(q_{y}(\gamma))=y\in U\subseteq U_{I_0}$, we deduce
from~\eqref{eq:norm control} that $|g(q_y(\gamma)) - g(\gamma)|=
|g_j(q_y(\gamma_j)) - g_j(\gamma_j)| < \varepsilon$.
\end{proof}

\begin{lemma}\label{lem:sheesh}
%Let $G$ be an \'etale groupoid and suppose that $\Ii$ is abelian.
For $y \in
G^{(0)}$, let $\lambda_y \colon C^*_r(G) \to \Bb(\ell^2(Gy))$ be the left regular
representation, let $P_{\Ii y}$ be the orthogonal projection of $\ell^2(Gy)$ onto
$\ell^2(\Ii y)$, and let $\Phi_y \colon C^*_r(G) \to
\Bb(\ell^2(\Ii y))$ be the map
\begin{equation}\label{eq:def Phi_y}
\Phi_y(a) = P_{\Ii y} \lambda_y(a) P_{\Ii y}.
\end{equation}
Fix $x \in G^{(0)}$ and let $Q$ be the orthogonal projection of $\ell^2(Gx)$ onto
$\ell^2(\Iicl x)$. Then for each $g \in C_c(G)$ and each $\varepsilon >
0$, there is a neighbourhood $U$ of $x$ such that $\|\Phi_y(g)\| \le \|Q
\lambda_x(g) Q\| + \varepsilon$ for all $y \in U$.
\end{lemma}
\begin{proof}
Fix $g \in C_c(G)$ and $\varepsilon > 0$. Let $N := |\Iicl x \cap \supp(g)|$.
By Lemma~\ref{lem:editing g}, there exists a neighbourhood $U$ of $x$ such
that for each $y \in U$ there are a subgroup $K_y$ of $\Iso(G) x$ contained
in $\Iicl x$ and a homomorphism $q_y \colon K_y \to \Ii y$ such that
\begin{enumerate}[label=\textup{(\alph*)}]
\item\label{it:lem editing g:subset,specific} $\Ii y \cap \supp^\circ (g) \subseteq q_y(K_y \cap
    \supp(g))$;
\item\label{it:lem editing g:injective,specific} $q_y$ is injective on $K_y \cap \supp(g)$; and
\item\label{it:lem editing g:epsilon,specific}  $|g(q_y(\gamma)) - g(\gamma)| < \varepsilon/(N+1)$ for all $\gamma
    \in K_y \cap \supp(g)$.
\end{enumerate}

Fix $y \in U$. We must show that $\|\Phi_y(g)\| \le \|Q \lambda_x(g) Q\| +
\varepsilon$. Writing $\lambda^{\Ii y}_{\beta}$ for the unitaries in the regular representation of $\Ii y$, we can write $\Phi_y$  as
\[
 \Phi_y(g)
    = \sum_{\beta  \in \Ii y} g(\beta )\lambda^{\Ii y}_\beta.
\]

In this proof, we denote the universal unitary representations of the groups $K_y$, $\Ii y$ and $q_y(K_y) \subseteq \Ii_y$ respectively by
\begin{gather*}
K_y \owns \gamma \mapsto W^{K_y}_\gamma \in C^*(K_y),
    \qquad
    \Ii_y \owns \beta \mapsto W^{\Ii y}_\beta \in C^*(\Ii_y),
    \qquad\text{and}\\
    q_y(K_y) \owns \beta \mapsto W^{q_y(K_y)}_\beta \in C^*(q_y(K_y)),
\end{gather*}
and we denote the regular representation of $K_y$ by
\[
    K_y \owns \gamma \mapsto \lambda^{K_y} \in \Bb(\ell^2(K_y)).
\]
Since $\Ii y$ is abelian and hence amenable,
\begin{equation}\label{eq:lambdaI->WI}
    \Big\|\sum_{\beta  \in \Ii y} g(\beta )\lambda^{\Ii y}_\beta \Big\|
    = \Big\|\sum_{\beta  \in \Ii y} g(\beta )W^{\Ii y}_\beta \Big\|.
\end{equation}
Using that $g(\beta) = 0$ for $\beta \in \Ii y \setminus
q_y(K_y\cap \supp(g))$ by Property~\ref{it:lem editing g:subset,specific}, we see that
\begin{equation}\label{eq:Iy->Ky cap supp}
\Big\|\sum_{\beta \in \Ii y} g(\beta)W^{\Ii y}_\beta\Big\|
    = \Big\|\sum_{\beta \in q_y(K_y\cap \supp(g))} g(\beta)W^{\Ii y}_\beta\Big\|.
\end{equation}
Since the inclusion $q_{y}(K_y) \hookrightarrow \Ii y$ of groups induces an injective, and hence isometric, homomorphism of full $C^*$-algebras
$C^*(q_y(K_y)) \hookrightarrow C^*(\Ii y)$, we can rewrite the right-hand side,
\begin{equation}\label{eq:WI->WqK}
\Big\|\sum_{\beta \in q_y(K_y\cap \supp(g))} g(\beta)W^{\Ii y}_\beta\Big\|
    = \Big\|\sum_{\beta \in q_y(K_y\cap \supp(g))} g(\beta)W^{q_y(K_y)}_\beta\Big\|.
\end{equation}
Using~\eqref{eq:lambdaI->WI}, \eqref{eq:Iy->Ky cap supp}~and~\eqref{eq:WI->WqK} at the first step, and then Property~\ref{it:lem editing g:injective,specific}
of the homomorphism $q_y$ at the second, we see that
\[
\|\Phi_y(g)\|
    = \Big\|\sum_{\beta \in q_y(K_y\cap \supp(g))} g(\beta) W^{q_y(K_y)}_\beta\Big\|
    = \Big\|\sum_{\gamma \in K_y \cap \supp(g)} g(q_y(\gamma)) W^{q_y(K_y)}_{q_y(\gamma)}\Big\|.
\]
Writing each $g(q_y(\gamma)) W^{q_y(K_y)}_{q_y(\gamma)} = g(\gamma) W^{q_y(K_y)}_{q_y(\gamma)} + \big(-g(\gamma) + g(q_y(\gamma))\big)W^{q_y(K_y)}_{q_y(\gamma)}$ and then applying the triangle inequality, we obtain
\begin{align*}
\|\Phi_y(g)\|
    &\le \Big\|\sum_{\gamma \in K_y \cap \supp(g)} g(\gamma) W^{q_y(K_y)}_{q_y(\gamma)}\Big\| + \sum_{\gamma \in K_y \cap \supp(g)} |g(\gamma)- g(q_y(\gamma))|.
\end{align*}
The homomorphism $q_y \colon K_y \to q_{y}(K_y)$ induces a homomorphism $\pi_y \colon C^*(K_y) \to C^*(q_{y}(K_y))$ such that
$\pi_y(W^{K_y}_\gamma) = W^{q_y(K_y)}_{q_y(\gamma)}$ for all $\gamma \in K_y$. This combined with Property~\ref{it:lem editing g:epsilon,specific} yields
\begin{align*}
\|\Phi_y(g)\|
    &\le \Big\|\pi_y\Big(\sum_{\gamma \in K_y \cap \supp(g)} g(\gamma) W^{K_y}_{\gamma}\Big)\Big\| + \sum_{\gamma \in K_y \cap \supp(g)} \varepsilon/(N+1).
\end{align*}
Since $\pi_y$ is a $C^*$-homomorphism, it is norm-decreasing, and so, since $N = |\Iicl x \cap \supp(g)| \ge |K_y \cap \supp(g)|$, we obtain
\begin{equation}\label{eq:Phiyg estimate}
\|\Phi_y(g)\|
    < \Big\|\sum_{\gamma \in K_y \cap \supp(g)} g(\gamma) W^{K_y}_{\gamma}\Big\| + \varepsilon
    = \Big\|\sum_{\gamma \in K_y} g(\gamma) W^{K_y}_\gamma\Big\| + \varepsilon.
\end{equation}
Since $K_y$ is abelian and hence amenable,
\begin{equation}\label{eq:full->regular}
\Big\|\sum_{\gamma \in K_y} g(\gamma) W^{K_y}_\gamma\Big\|
    = \Big\|\sum_{\gamma \in K_y} g(\gamma) \lambda^{K_y}_\gamma\Big\|.
\end{equation}

Let $Q_{K_y}$ be the orthogonal projection of $\ell^2(Gx)$ onto $\ell^2(K_y)$.  Recall that $\lambda_x$ is the regular
representation of $C^*(G)$ on $\ell^2(Gx)$. We claim that
\begin{equation}\label{anna:claim1,v2}
	\sum_{\gamma \in K_y} g(\gamma) \lambda^{K_y}_\gamma = Q_{K_y} \lambda_x(g) Q_{K_y},
\end{equation}
regarded as operators on $\ell^2(K_y)$. To see this, we fix $\eta \in K_y$ (so that $e_\eta$ is a typical basis element for $\ell^2(K_y)$), and show that $\sum_{\gamma \in K_y} g(\gamma) \lambda^{K_y}_\gamma e_\eta = Q_{K_y} \lambda_x(g) Q_{K_y} e_\eta$. We calculate
\begin{align*}
\bigl(Q_{K_y}\lambda_x(g)Q_{K_y}\bigr) e_\eta
	&=	Q_{K_y}\lambda_x(g) e_\eta
        &&\text{as } Q_{K_y}e_\eta = e_\eta \\
    &= \sum_{\gamma\in G_{r(\eta)}} g(\gamma)Q_{K_y} e_{\gamma\eta}
        &&\text{by definition of }\lambda_x \\
	&= \sum_{\gamma\in K_y} g(\gamma)e_{\gamma\eta}
        &&\text{as $K_y\leq G$, and by defnition of $Q_{K_{y}}$}\\
    &= \sum_{\gamma\in K_y} g(\gamma) \lambda^{K_y}_{\gamma} e_\eta
        &&\text{by definition of }\lambda^{K_y},
	\end{align*}
estabishing~\eqref{anna:claim1,v2}. Recall that $Q \colon \ell^2(Gx) \to \ell^2(\Iicl x)$ is the orthogonal projection. Since $\ell^2(K_y) \subseteq
\ell^2(\Iicl x)$, we have $Q_{K_y} \le Q$, and so~\eqref{eq:full->regular} gives
\begin{align*}		
\Big\|\sum_{\gamma \in K_y} g(\gamma) W^{K_y}_\gamma\Big\|
	&\overset{\eqref{eq:full->regular}}{=} \Big\|\sum_{\gamma \in K_y} g(\gamma) \lambda^{K_y}_\gamma\Big\|\\
	&\overset{\eqref{anna:claim1,v2}}{=} \Big\| Q_{K_y} \lambda_x(g) Q_{K_y}\Big\|		
    = \Big\|Q_{K_y} Q \lambda_x(g)Q Q_{K_y}\Big\|
    \leq \Big\| Q \lambda_x(g) Q\Big\|.
\end{align*}
Combining this with~\eqref{eq:Phiyg estimate} yields the desired estimate $\|\Phi_y(g)\| \le \|Q \lambda_x(g) Q\| + \varepsilon$.
\end{proof}

\begin{proof}[Proof of Theorem~\ref{thm:phew!}]
Fix $a \in C^*_r(G)$ such that the open support of $j(a)$ is contained in
$\Ii$ and $j(a)|_{\Ii x} = 0$. We prove that for every $\varepsilon
> 0$ there exists $h \in C_0(G^{(0)})$ such that $h(x) = 0$ and $\|a - ah\|
\le \varepsilon$; since each $ah \in J_x$ this proves that $a \in J_x$.

So
fix $\varepsilon > 0$.
Since $C_c(G)$ is dense in $C^*_r(G)$, there exists $g \in C_c(G)$ such that
\begin{equation}\label{eq:choice of g}
    \|g - a\| < \varepsilon/4.
\end{equation}

Let $Q \colon \ell^2(Gx) \to \ell^2(\Iicl x)$ be the orthogonal projection, and let $\Phi_y$ be as defined in Equation~\eqref{eq:def Phi_y}. By
Lemma~\ref{lem:sheesh} there is a neighbourhood $U$ of $x$ such that
\begin{equation}\label{eq:property of Phi_y in pf of thm}
    \sup_{y\in U}\|\Phi_y(g)\| < \|Q \lambda_x(g) Q\| + \varepsilon/4.
\end{equation}
Since $j(a)$ is identically zero on $\Ii x \cup G\setminus\Ii$ by assumption, we
have $j(a)(\gamma) = 0$ for all $\gamma \in Gx$. Hence, for $\gamma,\eta \in
\Iicl x$, since $\gamma\eta^{-1} \in Gx$, we have $( \lambda_x(a) e_\gamma \mid e_\eta ) = j(a)(\gamma\eta^{-1}) = 0$. Hence
\begin{align*}
( Q \lambda_x(a)Q e_\gamma \mid e_\eta )_{\ell^2(\Iicl x)}
    &= ( \lambda_x(a)Q e_\gamma \mid Q e_\eta )_{\ell^2(Gx)}
	 = 0.
\end{align*}
Hence $Q \lambda_x(a)Q e_\gamma = 0$ for each basis element $e_\gamma$, so $Q \lambda_x(a) Q = 0$. It follows from~\eqref{eq:choice of g} that
$\varepsilon/4 > \|Q\lambda_x(g)Q - Q\lambda_x(a)Q\| = \|Q \lambda_x(g) Q\|$, and thus by~\eqref{eq:property of Phi_y in pf of thm},
\begin{equation}\label{eq:sup over U estimate}
    \sup_{y \in U} \|\Phi_y(g)\| < \varepsilon/2.
\end{equation}

If $a \in A$ is nonzero, then $j(a)(\gamma) \not= 0$ for some $\gamma \in
\Ii$, and hence $( \Phi_{s(\gamma)}(a) e_{s(\gamma)} \mid e_\gamma) =
j(a)(\gamma)$ is nonzero. So $\bigoplus_{y \in G^{(0)}} \Phi_y$ is injective
on $A$ and hence
\begin{equation}\label{eq:oplus Phi isometric}
\Big\|\bigoplus_{y \in G^{(0)}} \Phi_y(b)\Big\| = \|b\|\qquad\text{for all $b \in A$.}
\end{equation}

Fix $h \in C_0(G^{(0)}, [0,1])$ such that $h(x) = 0$ and $h(y) = 1$ for all
$y \in s(\supp(g)) \setminus U$. for $y \in G^{(0)}$ and $\eta \in \Ii_y$, we have
\begin{align*}
\Phi_y(g(1-h)) e_\eta
    &= \sum_{\gamma \in G y} g(\gamma)(1 - h(s(\gamma))) P_{\Ii y} e_{\gamma\eta}\\
    &= (1 - h)(y) \sum_{\gamma \in G y} g(\gamma) P_{\Ii y} e_{\gamma\eta}
    = (1 - h)(y) \Phi_y(g) e_\eta.
\end{align*}
Consequently,
\begin{equation} \label{eq:scalar mult}
    \Phi_y(g(1-h)) = (1-h)(y) \Phi_y(g)\quad\text{ for all $y \in G^{(0)}$.}
\end{equation}

Using~\eqref{eq:oplus Phi isometric}, we have
\[
\|a - ah\|
    = \Big\|\Big(\bigoplus_{y \in G^{(0)}} \Phi_y\Big)(a - ah)\Big\|
	= \sup_{y \in G^{(0)}} \|\Phi_y(a - g + g - gh + gh - ah)\|.
\]
Applying the triangle inequality and then~\eqref{eq:choice of g} and that $\|h\| \leq 1$, we obtain
\begin{align*}
\|a - ah\| &\le \sup_{y \in G^{(0)}} \|\Phi_y(a - g)\| + \|\Phi_y(g - gh)\| + \|\Phi_y(a - g)\| \|\Phi_y(h)\|\\
	&\le \varepsilon/4 + \sup_{y \in G^{(0)}} \|\Phi_y(g(1 - h))\| + \varepsilon/4.
\end{align*}
Using that $\Phi_y (g)=0$ for $y \notin s(\supp(g))$ together with~\eqref{eq:scalar mult} and that $1 - h$ vanishes off $U$, we obtain
\[
\|a - ah\| \le \sup_{y \in U} \|\Phi_y(g)\| + \varepsilon/2,
\]
and so \eqref{eq:sup over U estimate} gives $\|a - ah\| < \varepsilon$.
\end{proof}

\end{document}